\documentclass[a4paper,12pt]{amsart}
\usepackage{hyperref}

\usepackage{microtype}
\usepackage{amssymb,amsmath,amsfonts,amsthm,bm,latexsym,amscd,verbatim,url,nicefrac,stmaryrd,enumerate,colonequals,dsfont,appendix,stackrel,mathrsfs,cleveref}
\usepackage[all]{xy}

\usepackage{times}
\usepackage{graphicx}
\usepackage{fullpage}

\makeatletter
\newcommand{\oset}[3][0ex]{%
	\mathrel{\mathop{#3}\limits^{
			\vbox to#1{\kern-2\ex@
				\hbox{$\scriptstyle#2$}\vss}}}}
\makeatother

\newtheorem{theorem}{Theorem}{\bf}{\it}
\newtheorem*{thm*}{Theorem}
{\bf}{\it}
\newtheorem{prop}[theorem]{Proposition}

\newtheorem{cor}[theorem]{Corollary}

\theoremstyle{definition}
\newtheorem{dfn}[theorem]{Definition}
\theoremstyle{remark}
\newtheorem{rmk}[theorem]{Remark}
\theoremstyle{remark}

\newcommand{\dRR}[2]{
	{\rm R\Gamma_{\dR}}
	(#1/#2)
}

\newcommand{\A}{\mathbb{A}}

\newcommand{\G}{\mathbb{G}}

\newcommand{\Q}{\mathbb{Q}}
\newcommand{\RR}{\mathbb{R}}

\newcommand{\mcF}{\mathcal{F}}

\newcommand{\mcD}{\mathcal{D}}

\newcommand{\mcO}{\mathcal{O}}

\newcommand{\mcU}{\mathcal{U}}

\newcommand{\mfm}{\mathfrak{m}}

\DeclareMathOperator{\Aff}{Aff}

\DeclareMathOperator{\colim}{colim}

\DeclareMathOperator{\dR}{dR}

\DeclareMathOperator{\eff}{eff}
\DeclareMathOperator{\et}{\acute{e}t}

\DeclareMathOperator{\gm}{gm}

\DeclareMathOperator{\op}{{op}}

\DeclareMathOperator{\Perf}{Perf}

\DeclareMathOperator{\Sm}{Sm}

\DeclareMathOperator{\Spec}{Spec}

\DeclareMathOperator{\uhom}{\underline{Hom}}

\DeclareMathOperator{\fd}{fd}

\DeclareMathOperator{\DA}{{DA}}

\DeclareMathOperator{\QCoh}{{QCoh}}

\begin{document}
	\title{A motivic proof of the finiteness of  relative de Rham cohomology}

\author[Vezzani]{Alberto Vezzani}
\address{Dipartimento di Matematica ``F. Enriques'' - Universit\`a degli Studi di Milano}
\email{alberto.vezzani@unimi.it}
\urladdr{users.mat.unimi.it/users/vezzani/}

\begin{abstract}
We give a  quick proof of the fact that the relative de Rham cohomogy groups $H^i_{\dR}(X/S)$ of a smooth  and proper map $X/S$ between schemes over $\Q$ are vector bundles on the base, replacing  Hodge-theoretic  and transcendental methods with $\A^1$-homotopy theory.
\end{abstract}

\maketitle

It is well known that the relative algebraic de Rham cohomology groups $H^i_{\dR}(X/S)$  define  vector bundles on $S$ whenever $X/S$ is a smooth and proper map of schemes over $\Q$. If $S$ is a smooth variety, this can be deduced from the fact that such groups have a natural structure of modules with connection (\cite{katz}, \cite[Section 2.2]{andre-ens}). In the general case, this is a result by Deligne \cite{deligne} which uses the degeneracy of the relative Hodge-de Rham spectral sequence, a reduction step to the case in which $S$ is the spectrum of a local artinian ring, and a further reduction to the case of a field using the invariance of de Rham under infinitesimal thickenings  (see also \cite[Lectures 9 and 10]{clausen}).  %

We now show that the result can be deduced formally from the $\A^1$-invariance of the de Rham cohomology. Indeed, all the  reduction steps in the original proof are ``motivic'' and can be condensed in a few lines (see the proof of \Cref{main}) at the expense of using some cornerstone results of motivic homotopy theory (the  localization triangle, the dualizability of smooth and proper maps). This proof  has the advantage of being independent on the specifics of the de Rham realization\footnote{For instance, we do not need to introduce the complex numbers.} (any other realization in quasi-coherent modules would do) and may be adapted to other contexts too.
\begin{dfn}
	\label{definition-of-omega-1}
	Let $f: X \to S$ be a smooth morphism of schemes. 
	We let $\dRR{X}{S}$ be the complex of $\mcO_S$-modules $\RR f_*\Omega^\bullet_{X/S}$.
\end{dfn}

The following proposition is just a collection of well-known facts \cite[Tag 0FK4]{stacks}.

\begin{prop}
	Let $f\colon X=\Spec  B\to S=\Spec A$ be a smooth map of affine schemes and $g\colon S'=\Spec A'\to \Spec A$ be another map.
	\begin{enumerate}
\item\label{1} Each $\Omega^d_{X/S}$ is a projective module on $X$ and $H^i_{\et}(X,\Omega^d_{X/S})=0$. 
\item\label{2} There is a canonical equivalence $g^*\dRR{X}{S}\simeq \dRR{X\times_SS'}{S'}$.
\item \label{3}If $g$ is also smooth, there is a canonical equivalence $\dRR{X}{S}\otimes_{\mcO_S}\dRR{S'}{S}\simeq \dRR{X\times_SS'}{S}$.
\item \label{4}Let $S$ be over $\Q$. If $X=\A^1_S$ then $\dRR{X}{S}\simeq\mcO_S[0]$ and if $X=\G_{m,S}$ then $\dRR{X}{S}\simeq\mcO_S[0]\oplus\mcO_S[-1]$. 
	\end{enumerate}
\end{prop}
We fix a scheme $S$. As in \cite{ayoub-etale,agv} we introduce the category of \emph{effective \'etale rational motives} $\DA(S)^{\eff}\colonequals\DA_{\et}^{\eff}(S,\Q)=\mathrm{SH}_{\et}^{\eff}(S,\Q)$ as being the full  subcategory of the  infinity-category of \'etale $\Q$-sheaves on the category $\Aff\Sm/S$ of affine smooth schemes over $S$, whose objects $\mcF$ are  \emph{$\A^1$-local} (in the sense that $\Gamma(X,\mcF)\simeq\Gamma(\A^1_X,\mcF)$ for any $\Aff\Sm/S$), cfr. \cite[Definition 2.1.11]{agv}.  %
Typical objects of this category are the $\A^1$-localizations of the ($\Q$-enriched) sheaves represented by  $X\in\Sm/S$, denoted by $\Q_S(X)$.  This is the universal $\Q$-linear presentable infinity-category over $\Aff\Sm/X$ for which the following equivalences hold: $\Q_S(X)\simeq\Q_S(\A^1_X)$, and $\Q_S(X)\simeq\colim\Q_S(\mcU)$ for any \'etale truncated hypercover $\mcU\to X$ (cfr. \cite[Remark 2.1.19]{agv}). It is a presentable symmetric monoidal infinity-category for which $\Q_S(X)\otimes\Q_S(Y)\simeq\Q_S(X\times_SY)$ \cite[Proposition 2.2.1.9]{HA}.%

We let $\DA(S)$ be the infinity category $\varprojlim\DA(S)^{\eff}$ with transition maps given by $\mcF\mapsto\uhom(T^1_S,\mcF)$ where $T_S^1$ is the cofiber of the (split!) inclusion  $\Q_S(S)\to\Q_S(\G_{m,S})$ given by the unit section. In this the universal presentable monoidal category over $\DA^{\eff}(S)$ on which   the endofunctor $\mcF\mapsto\mcF\otimes T^1_S$ is invertible (see  \cite[Definition 2.38]{robalo}). %

\begin{rmk}
The infinity categories $\DA(S)$ %
are equipped with a six-functor formalism \cite{ayoub-th1, ayoub-th2} and \cite[Section 3]{ayoub-ICM}. In particular,  any  morphism $f\colon S'\to S$ induces a (monoidal) pullback functor $f^*\colon\DA(S)\to\DA(S')$. They also satisfy \'etale descent \cite[Theorem 2.3.4]{agv}.%
\end{rmk}

If $S$ is an affine scheme of  finite dimension, then $\DA(S)$ is compactly generated by twists and shifts of the motives $\Q_S(X)$ with $X$ in $\Aff\Sm/S$ \cite[Remark 2.4.23]{agv} and we let $\DA^{\gm}(S)$ be the full subcategory of compact objects in this case. Note that the pullback functors (along maps between affine schemes) preserve the categories $\DA^{\gm}(S)$. If $S$ is a general scheme, we let $\DA(S)^{\gm}$ be the full stable  subcategory of $\DA(S)$  obtained by (Zariski) descent from the categories $\DA(S)^{\gm}$ for $S$ affine. As the unit object is compact in $\DA(S)$ whenever $S$ is affine, this category contains the  subcategory of (fully) dualizable objects $\DA(S)^{\fd}$. %
 \begin{rmk}One may consider alternatively the category $\DA^{\wedge}_{\et}(S,\Q)$ (see \cite[Page 6]{agv}) which is defined similarly out of \'etale \emph{hypersheaves} (it is the one typically considered in the motivic literature). It enjoys many of the formal properties of $\DA(S)$ (it may not be compactly  generated though). As a matter of fact, it coincides with $\DA(S)$ whenever $S$ is locally of finite Krull dimension (see \cite[Proposition 2.4.19]{agv}). 
 \end{rmk}
 As in \cite{lurie:SAG}, we denote by $\QCoh(S)$ the derived infinity-category of quasi-coherent modules over a scheme $S$ and by  $\Perf(S)$ its full subcategory of (locally) perfect complexes.
 
\begin{cor}Let $S$ be a scheme over $\Q$.\begin{enumerate}[(a)]
\item\label{b} The functor $X\mapsto \dRR{X}{S}$ defined on affine schemes can be extended uniquely to a functor of infinity-categories $\DA(S)^{\gm}\to\QCoh(S)^{\op}$ for any $S$,  which is  monoidal and compatible with arbitrary pullbacks. %
\item \label{c}The functor $\dR_S$ restricts to a functor $\DA(S)^{\fd}\to\Perf(S)^{\op}$.
\end{enumerate}
\end{cor}
\begin{proof}
If $S$ is affine, by \eqref{1} and \eqref{4}  we deduce that $\dRR{-}{S}$ defines a complex of presheaves on $\Aff\Sm/S$ which is \'etale local (in the sense that $\Gamma(X,\mcF)\simeq\mathrm{R}\Gamma_{\et}(X,\mcF)$ for any $X\in\Aff\Sm/S$) and $\A^1$-local, hence an object of $\DA^{\eff}(S)$ (even, of $\DA^{\wedge}_{\et}(S,\Q)$). The values of $X\mapsto\dRR{X}{S}$ lie in $\QCoh(S)$ hence defining a functor $\DA^{\eff}(S)\to\QCoh(S)^{\op}$. By \eqref{4} it sends $T^1_S$ to an invertible object and therefore (using \eqref{3}) it extends canonically to $\DA(S)$.  
Properties \eqref{2} and  \eqref{3}  imply the statement of \eqref{b} for $S$ affine.  The general case can be obtained by Zariski descent of both categories involved. %

Property \eqref{c} follows formally from \eqref{b} by taking dualizable objects on both sides (see e.g. \cite[Proposition 7.2.4.4]{HA} for a description of dualizable objects in the target).%
\end{proof}

\begin{rmk}One can define a de Rham realization on the whole $\DA(S)$ for any $S$, compatible with pullbacks along smooth morphisms, at the expense of considering the (derived category of) \emph{solid} modules $\mcD(\mcO_{S,\blacksquare})$ as a target as it is done in \cite[Corollary 4.39]{lbv}. Solid modules  have indeed a full six-functor formalism \cite[Lecture XI]{scholze-cond}  so that the base change$f^*$ along a smooth map $f$  is equipped with a left adjoint $f_!$ and therefore commutes with limits, which allows one to extend the equivalence $f^*\dR\simeq\dR f^*$ from $\DA(S)^{\gm}$ to $\DA(S)$. Note that the category of dualizable objects in $\mcD(\mcO_{S,\blacksquare})$ coincides with $\Perf(S)$ so that the theorem below would still follow as stated. It is also possible to give a  \emph{covariant} version of the de Rham realization using relative de Rham stacks. We do not pursue this approach here.
\end{rmk}
\begin{theorem}\label{main}
	Let $S$ be a scheme over $\Q$  and let $M$ be a dualizable motive in $\DA(S)$. Then $\dRR{M}{S}$ is a perfect complex whose cohomology groups are vector bundles on $S$. 
\end{theorem}
\begin{proof}
The statement is \'etale local on $S$ so we may assume it is affine over $\bar{\Q}$. We may write $\mcO(S)$ as a union of its finitely generated $\bar{\Q}$-subalgebras $A$ and $\DA(S)^{\gm}\simeq\varinjlim\DA(\Spec A)^{\gm}$ (\cite[Lemma 1.A.2]{ayoub-rig} and \cite[Proof of 2.5.1]{agv}). We can then find a dualizable model $M_A$ of $M$ over some  $A$ as above. Since $\dRR{M}{S}\simeq\dRR{M_A}{\Spec A}\otimes_{A}\mcO(S)$ we may and do replace $S$ with $\Spec A$ and assume that it is of finite type over $\bar{\Q}$.  We already know that $\dRR{M}{S}$ is a perfect complex of $A$-modules using \eqref{c}. We are now trying to prove that its  cohomology groups (which are coherent $A$-modules) are projective, or equivalently that this complex splits. Using e.g. \cite[Lemmas 00MA and 00MC]{stacks} and \cite[Théorème  II.5.2.1]{bourbaki-algcomm}, it suffices to work on the completed stalks at ($\bar{\Q}$-rational) maximal points, and hence to  show that the perfect complex $\dRR{M}{S}\otimes_{\mcO_S}\widehat{\mcO}_{S,s}$ is split for any closed point $s$ of $S$.
   
We fix one of them, and we let   $\DA(\widehat{ S}_s)$ be the infinity-category $\varprojlim\DA(\Spec(\mcO_{S,s}/\mfm^n)) $. We point out that this category is equivalent to $\DA(\bar{\Q})$ as each functor $\DA(\Spec(\mcO_{S,s}/\mfm^n))\to\DA(\bar{\Q})$ is invertible \cite[Remarque 2.1.161]{ayoub-th1}. We also let $M\mapsto\widehat{ M}_s$ be the natural functor $\DA(S)\to\DA(\widehat{ S}_s)$. It corresponds, under the equivalence $\DA(\widehat{ S}_s)\simeq\DA(\bar{\Q})$ to the functor $M\mapsto M_s=s^*M$. The de Rham realization  induces a functor $\dRR{-}{\widehat{\mcO}_{S,s}}\colon\DA(\widehat{ S}_s)^{\fd}\to\varprojlim\Perf(\mcO_{S,s}/\mfm^n)^{\op}\simeq\Perf(\widehat{ \mcO}_{S,s})^{\op}$ \cite[Proposition A.4]{efimov}.  
 If we denote by  $\Pi $ be the structural morphisms to $\Spec\bar{\Q}$, we have that $s^*\widehat{ M}_s$ and $s^*\Pi^*M_{s}$ coincide on $\DA(\bar{\Q})$ and hence $ \widehat{ M}_s\simeq\Pi^*M_{s}$ in $\DA(\widehat{ S}_s)$ so that $\dRR{M}{S}\otimes_{\mcO_S}\widehat{\mcO}_{S,s}\simeq\dRR{\widehat{ M}_s}{\widehat{ \mcO}_s}\simeq\dRR{M_s}{\bar{\Q}}\otimes_{\bar{\Q}}\widehat{ \mcO}_{S,s}$. This latest complex is a \emph{split} perfect complex of $\widehat{ \mcO}_{S,s}$-modules since $\dRR{M_s}{\bar{\Q}}$ obviously is a split complex of $\bar{\Q}$-vector spaces, and this concludes the proof.
\end{proof}

 \begin{cor}\label{cor}
If $X/S$ is a smooth and proper map of schemes over $\Q$, then the de Rham cohomology groups $H^i_{\dR}(X/S)$ are vector bundles over $S$.
 \end{cor}
\begin{proof}
If $X/S$ is smooth and proper, then $\Q_S(X)$ is dualizable \cite[Th\'eor\`eme 2.2]{riou-dual}. Note that $\dR(\Q_S(X))$ coincides with the ``classical'' complex computing relative de Rham cohomology, as it enjoys Zariski (even: \'etale) descent on source and target.
\end{proof}

\begin{rmk}
The proof above has been inspired by the proof of  \cite[Theorem 4.46]{lbv} which deals with adic spaces over a non-archimedean field of characteristic zero. There are differences between the two situations though: for instance in the analytic setting, the ($\pi$-adically completed) stalk of the structure sheaf at a rational point is directly a (valued) field (see \cite[Theorem 2.8.6]{agv}).
\end{rmk}

\bibliographystyle{plain}

\end{document}